\documentclass[12pt]{article}
\usepackage{float}
\usepackage{makeidx}
\usepackage{latexsym}

\usepackage[margin=2.5cm]{geometry}
\usepackage[dvips]{graphicx}
\usepackage[english]{babel}
\usepackage{amssymb}
\usepackage{amsmath}
\usepackage{amsthm}
\usepackage{lineno}
\usepackage{xcolor}
\definecolor{blau}{rgb}{0.1,0.0,0.9}
\definecolor{gruen}{cmyk}{1.0,0.2,0.7,0.07}
\definecolor{mag}{cmyk}{0.0,0.9,0.3,0.0}

\newtheorem{theorem}{Theorem}[section]
\newtheorem{lemma}[theorem]{Lemma}

\newtheorem{property}[theorem]{Property}
\theoremstyle{definition}
\newtheorem{problem}[theorem]{Problem}
\newtheorem{definition}[theorem]{Definition}

\begin{document}
\date{\today}
\title{Latin cubes with forbidden entries}

\author{
{\sl Carl Johan Casselgren}\footnote{Department of Mathematics, Link\"oping University, SE-581 83 Link\"oping, Sweden. \newline
{\it E-mail address:} carl.johan.casselgren@liu.se  }
\and {\sl Klas Markstr\"om}\footnote{Department of Mathematics, 
Ume\aa\enskip University, 
SE-901 87 Ume\aa, Sweden.
{\it E-mail address:} klas.markstrom@umu.se
}
\and {\sl Lan Anh Pham }\footnote{Department of Mathematics, 
Ume\aa\enskip University, 
SE-901 87 Ume\aa, Sweden.
{\it E-mail address:} lan.pham@umu.se
}
}
\maketitle

\bigskip
\noindent
{\bf Abstract.}
We consider the problem of constructing Latin cubes subject to the condition
that some symbols may not appear in certain cells. 
We prove that there is 
a constant $\gamma > 0$ such that if $n=2^k$ and $A$ is
$3$-dimensional $n\times n\times n$ array where every cell contains at most
$\gamma n$ symbols, and every symbol occurs at most $\gamma n$ times in 
every line of $A$, then $A$ is {\em avoidable}; that is, there is
a Latin cube $L$ of order $n$ such that for every $1\leq i,j,k\leq n$,
the symbol in position $(i,j,k)$ of $L$ does not appear in the corresponding
cell of $A$.

\bigskip


\section{Introduction}
Consider an $n \times n$ array $A$
in which 
every cell $(i,j)$
contains a subset $A(i,j)$ of the symbols in $[n]=\{1, \dots,n \}$.
If every cell contains at most $m$ symbols, and every symbol occurs
at most $m$ times in every row and column, then $A$ is an {\em $(m,m,m)$-array}.
Confirming a conjecture by H\"aggkvist \cite{Haggkvist}, it was proved in
\cite{AndrenCasselgrenOhman} that there is a constant $c>0$ such that if
$m \leq cn$ and $A$ is an $(m,m,m)$-array, then $A$
 is {\em avoidable};
that is, there is a Latin square $L$ such that for every $(i,j)$ the symbol in
position $(i,j)$ in $L$ is not in $A(i,j)$ 
(see also \cite{Lina, AndrenCasselgrenMarkstrom}).
The purpose of this note is to prove an analogue of this result for Latin cubes
of order $n=2^k$.

In order to make this precise, we imagine a $3$-dimensional array having
layers stacked on top of each other; we shall refer to
such a $3$-dimensional array as a {\em cube}. 
Now, a cube has {\em lines} in three directions
obtained from fixing two coordinates
and allowing the third to vary. The lines obtained by varying the 
first, second, and third coordinates will be 
referred respectively as {\em columns}, {\em rows}, and {\em files}. 
The first,
second, and third coordinates themselves will be referred to 
as the indices of the rows, columns, and files.

A \textit{Latin cube} $L$ of order $n$ on the symbols $\{1,\dots,n\}$ 
is an $n \times n \times n$ cube  such that each symbol in $\{1,\dots,n\}$ appears
exactly once in each row, column and file. The symbol in position $(i,j,k)$
of $L$ is denoted by $L(i,j,k)$.    Latin cubes have been studied by a number of authors, both with respect to enumeration 
and e.g. extension from partial cubes.  An extensive survey of early results can be found in \cite{MckWan}.

An  $n \times n \times n$ cube where 
each cell contains a subset of the symbols in the set $\{1,\dots,n\}$ is called an 
\textit{$(m,m,m,m)$-cube (of order $n$)} if the following conditions are 
satisfied:
\begin{itemize}
\item[(a)] No cell contains a set with more than $m$ symbols. 
\item[(b)] Each symbol occurs at most $m$ times in each row.
\item[(c)] Each symbol occurs at most $m$ times in each column.
\item[(d)] Each symbol occurs at most $m$ times in each file.
\end{itemize}

Let $A(i,j,k)$ denote the set of symbols in the cell $(i,j,k)$ of $A$. 
If we simplify notation, and write $A(i,j,k)=q$ if the set of 
symbols in cell $(i,j,k)$ of $A$ is $\{q\}$, then 
a $(1,1,1,1)$-cube is a {\em partial Latin cube},
and a \textit{Latin cube} $L$ is simply a $(1,1,1,1)$-cube
with no empty cell.

Given an $(m,m,m,m)$-cube $A$ of order $n$, a Latin cube $L$
of order $n$ \textit{avoids} $A$ if
there is no cell $(i,j,k)$ 
of $L$ such that $L(i,j,k) \in A(i,j,k)$;
if there is such a Latin cube, then
$A$ is \textit{avoidable}. 

Problems on extending partial Latin cubes have been studied for a long time, with the earliest results appearing in the 1970s \cite{Cruse}; in the more recent literature we have  \cite{Bri, Bryant,KuhlDenley,DenleyOhman}. The more general problem of constructing Latin cubes subject to the condition that some
symbols cannot appear in certain cells seems to be a hitherto quite unexplored
line of research.  Our main result is the following, which establishes an analogue of
the main result of \cite{Lina}, which considered Latin squares,  for Latin cubes.
\begin{theorem}
\label{maintheorem}
There is a positive constant $\gamma$ such that if $t \geq 30$ and 
$m \leq \gamma 2^t$, then any $(m,m,m,m)$-cube $A$ of order $2^t$ is avoidable.
\end{theorem}

The restriction on the order of the cube is not believed to be necessary, but as for Latin squares,  general orders are expected to require  far more technical proof
(unless  some completely new method is invented).  
Our proof establishes this result for a small value of $\gamma$ which we believe  to be far from the optimal one, much like 
the case for the similar results for Latin squares.   We know from 
\cite{CulterOhman} that $\gamma \leq \frac{1}{3}$, since that is an upper bound for 
the corresponding result for Latin squares, and every $n\times n$ 
sub-array of an  avoidable  $(m,m,m,m)$-cube of order $n$ must be avoidable
(in the sense that there is an $n\times n$ Latin square that avoids this array).
An interesting  question is  how sparse an unavoidable  $(m,m,m,m)$-cube can be if every square $n \times n$ sub-array is avoidable.

\begin{problem}
\label{problem}
For	how small $\gamma'=\frac{m}{n}$ does there exist an unavoidable  $(m,m,m,m)$-cube $A$ of order $n$,  where every square sub-array 
of order $n$  is avoidable for Latin squares?
\end{problem}

We note that $\gamma' \leq 1/2$, since there
are unavoidable $(n/2,n/2,n/2,n/2)$-cubes of order
$n$ that satisfies the condition in Problem \ref{problem}; such a cube can
be obtained by dividing an $n\times n\times n$ cube into $8$ subcubes of
equal order $n/2$, and putting symbols $1,\dots, n/2$ in all cells of two
subcubes in ``opposite'' corners of the larger cube.

We may also note that the main result of this paper, as well as the problem of extending partial Latin cubes, can be recast as list edge coloring problems on the 
complete $3$-uniform $3$-partite hypergraph $K^3_{n,n,n}$.  Problems on extending partial edge colourings for ordinary graphs have been studied to some
extent, see e.g. \cite{EGHKPS, GiraoKang} and the references given there,  
but similar problems for hypergraphs remain mostly unexplored.

In Section 2 we give some definitions and preparatory lemmas,
and in Section 3 we prove Theorem \ref{maintheorem}.


\section{Definitions and properties of Boolean Latin cubes}

In this section we give some definitions and collect essential
properties of Boolean Latin cubes.

Let $A$ be an $n \times n \times n$ cube.
Given $i \in [n]$, \textit{row layer} $i$ in 
$A$ is a set of $n^2$ cells $\{(i,j^*,k^*) : j^* \in [n], k^* \in [n] \}$;
given $j \in [n]$, \textit{column layer} $j$ in 
$A$ is a set of $n^2$ cells $\{(i^*,j,k^*) : i^* \in [n],k^* \in [n]\}$;
given $k \in [n]$, \textit{file layer} $k$  in $A$
is a set of $n^2$ cells $\{(i^*,j^*,k) : i^* \in [n], j^* \in [n]\}$. 
As mentioned above, by fixing two coordinates and varying the third,
we obtain rows, columns and files of a $n \times n \times n$ cube.
Formally we define a row of such a cube $A$ as a set of cells
$R_{i,k} = \{(i,j^*,k) : j^* \in [n]\}$, a column as the set
$C_{j,k} = \{(i^*,j,k) : i^* \in [n]\}$, and files
$F_{i,j} = \{(i,j,k^*) : k^* \in [n]\}$.

\begin{definition} 
The {\em Boolean Latin square} of order 
$2^t$ is the Latin square with entries as in the 
addition table of $\mathbb{Z}^t_2$
with the elements of $\mathbb{Z}^t_2$ 
mapped to the integers $1,\dots,2^t$. 
\end{definition}

A {\em $4$-cycle} in a Latin square $L$ is a set of four cells 
$\{(i_1,j_1), (i_1,j_2), (i_2,j_1), (i_2,j_2)\}$ such that 
$L(i_1,j_1)=L(i_2,j_2)$ and $L(i_1,j_2)=L(i_2,j_1)$.
We note some important properties of Boolean Latin squares (cf. \cite{Lina}).

\begin{property}
\label{prop:cycles}
Each cell in the $n \times n$
Boolean Latin square is in $n-1$ distinct $4$-cycles. Permuting the rows, the columns
or the symbols does not affect the number of $4$-cycles that a cell is part of.
\end{property}

\begin{property}
\label{pro4cyc}
A $4$-cycle in the Boolean Latin square
is uniquely determined by two cells; that is, if $C$ is a $4$-cycle and 
$(i_1,j_1), (i_1,j_2) \in C$, then
$(i_2,j_1), (i_2,j_2) \in C$, where $i_2$ is the row such that
$L(i_1,j_1) = L(i_2,j_2)$ 
and $L(i_1,j_2)= L(i_2,j_1)$.
\end{property}

\begin{property}
\label{two4cycle}
The intersection of two $4$-cycles is either empty, or it contains $1$ or $4$ cells.
\end{property}

Given an integer $t$, let $a_i$ ($1 \leq i \leq 2^t$) be the $i$th smallest element 
of  $\mathbb{Z}^t_2$. 
(For example, with $t=2$, $a_1=00, a_2=01, a_3=10, a_4=11$.)
We define the \textit{Boolean Latin cube} similarly as the Boolean Latin square.
\begin{definition}
\label{defBc}
The {\em Boolean Latin cube} $B$ of order $n=2^t$ on the symbols $\{1,\dots,n\}$ 
is an $n\times n \times n$ Latin cube such that
$B(i,j,k)=x$ with $a_x=a_i+a_j+a_k$ (addition in $\mathbb{Z}^t_2$) for all $1\leq i,j,k\leq n$.
\end{definition}

\begin{definition}
A {\em $3$-cube} in a Latin cube $L$ is a set of eight cells 
$$\{(i_1,j_1,k_1), (i_1,j_2,k_1), (i_2,j_1,k_1), (i_2,j_2,k_1), 
(i_1,j_1,k_2), (i_1,j_2,k_2), (i_2,j_1,k_2), (i_2,j_2,k_2)\}$$
such that $$L(i_1,j_1,k_1)=L(i_2,j_2,k_1)=L(i_1,j_2,k_2)=L(i_2,j_1,k_2)$$
and $$L(i_1,j_2,k_1)= L(i_2,j_1,k_1)=L(i_1,j_1,k_2)= L(i_2,j_2,k_2).$$
\end{definition}

Note that every row, column and file layer of the Boolean Latin cube
is a Boolean Latin square.
For the Boolean Latin cube we have the following analogue of 
Property \ref{prop:cycles}.

\begin{property}
Each cell in the Boolean Latin cube of order $n$ belongs to $n-1$ $3$-cubes. 
\end{property}
\begin{proof}
Consider an arbitrary cell $(i_1,j_1,k_1)$ of the Boolean Latin cube  $B$ which belongs to 
a $4$-cycle $\mathfrak{c}_1=\{(i_1,j_1,k_1), $ $ (i_1,j_2,k_1), $ $ (i_2,j_1,k_1), $ $ (i_2,j_2,k_1)\}$ such that
$B(i_1,j_1,k_1)=B(i_2,j_2,k_1)$ and $B(i_1,j_2,k_1)= B(i_2,j_1,k_1)$. 
There are $n-1$ $4$-cycles $\mathfrak{c}_1$ 
in file layer $k_1$
containing $(i_1,j_1,k_1)$, since 
by construction, the file layers of the Boolean 
Latin cube are isomorphic to Boolean Latin squares; this also holds
for row and column layers. 

Now, by Property
\ref{two4cycle}, the two cells $(i_1,j_1,k_1)$ and $(i_2,j_1,k_1)$ 
define a unique $4$-cycle
$$\mathfrak{c}_2=\{(i_1,j_1,k_1), (i_2,j_1,k_1), (i_1,j_1,k_2), (i_2,j_1,k_2)\}$$ 
in the column layer $j_1$
such that $B(i_1,j_1,k_1)=B(i_2,j_1,k_2)$ and $B(i_2,j_1,k_1)= B(i_1,j_1,k_2)$. 
By Definition \ref{defBc}, 
$$a_{i_1}+a_{j_1}+a_{k_1}= a_{i_2}+a_{j_2}+a_{k_1} =a_{i_2}+a_{j_1}+a_{k_2}$$ 
and 
$$a_{i_1}+a_{j_2}+a_{k_1}= a_{i_2}+a_{j_1}+a_{k_1}= a_{i_1}+a_{j_1}+a_{k_2}.$$
Hence, we have 
$$a_{i_1}+a_{j_1}+a_{k_2}= a_{i_2}+a_{j_2}+a_{k_2} =a_{i_2}+a_{j_1}+a_{k_1}$$
and 
$$a_{i_1}+a_{j_2}+a_{k_2}= a_{i_2}+a_{j_1}+a_{k_2}= a_{i_1}+a_{j_1}+a_{k_1};$$
or, in other words, 
$$B(i_1,j_2,k_1)= B(i_2,j_1,k_1)=B(i_1,j_1,k_2)= B(i_2,j_2,k_2)$$
and $$B(i_1,j_1,k_1)=B(i_2,j_2,k_1)=B(i_1,j_2,k_2)=B(i_2,j_1,k_2).$$ 
This implies that
$$\{(i_1,j_1,k_1), (i_1,j_2,k_1), (i_2,j_1,k_1),  (i_2,j_2,k_1), (i_1,j_1,k_2), 
(i_1,j_2,k_2), (i_2,j_1,k_2),  (i_2,j_2,k_2)\}$$ is a $3$-cube;
and so each cell in the Boolean Latin cube belongs to 
$n-1$ $3$-cubes. 
\end{proof}

\begin{property}
\label{traversalsetproperty}
Let $(i_1,j_1,k_1)$, $(i_2,j_2,k_2)$, $(i_3,j_3,k_3)$ be three cells
in the Boolean Latin cube $B$ such that
$(i_1 -i_2)(j_1-j_2)(k_1-k_2) \neq 0$, $(i_1 -i_3)(j_1-j_3)(k_1-k_3) \neq 0$
and $(i_2 -i_3)(j_2-j_3)(k_2-k_3) \neq 0$. 
If $(i_1,j_1,k_1)$ and $(i_2,j_2,k_2)$
both are in a $3$-cube $\mathcal{C}_1$, and $(i_1,j_1,k_1)$ and $(i_3,j_3,k_3)$
are in a $3$-cube $\mathcal{C}_2$,
then $(i_2,j_2,k_2)$ and $(i_3,j_3,k_3)$ are in a $3$-cube $\mathcal{C}_3$.
\end{property}

\begin{proof}
Assume $B(i_1,j_1,k_1)=x$, $B(i_2,j_2,k_2)=y$, $B(i_3,j_3,k_3)=z$. Since 
$\mathcal{C}_1$ and $\mathcal{C}_2$ are $3$-cubes, 
we have that $B(i_2,j_2,k_1)=B(i_1,j_1,k_1)=B(i_3,j_3,k_1)$, i.e,
 $a_{i_2} + a_{j_2} + a_{k_1} = a_{i_1} + a_{j_1} + a_{k_1} = a_{i_3} + a_{j_3} + a_{k_1}$.
It follows that $a_{i_2} + a_{j_2}= a_{i_3} + a_{j_3}$, 
so $a_{i_2}+a_{j_2} + a_{k_2} = a_{i_3} +a_{j_3} + a_{k_2}$, which implies that
$B(i_3,j_3,k_2)=B(i_2,j_2,k_2)=y$.
Similarly, we have $B(i_3,j_2,k_3)=B(i_2,j_3,k_3)=B(i_2,j_2,k_2)=y$ and 
$B(i_3,j_2,k_2)=B(i_2,j_3,k_2)=B(i_2,j_2,k_3)=B(i_3,j_3,k_3)=z$, which implies that
$(i_2,j_2,k_2)$ and $(i_3,j_3,k_3)$ are two cells of a $3$-cube 
$$\mathcal{C}_3=\{(i_2,j_2,k_2), (i_2,j_3,k_2), (i_3,j_2,k_2), (i_3,j_3,k_2),
 (i_2,j_2,k_3), (i_2,j_3,k_3), (i_3,j_2,k_3),  (i_3,j_3,k_3)\}.$$
\end{proof}

\begin{property}
The intersection of two $3$-cubes in a Latin cube
is either empty, or it contains $1$ or $8$ cells.
\end{property}
\begin{proof}
Assume that the intersection of two given $3$-cubes contains at least $2$ cells. 
If these $2$ cells lie in a $4$-cycle of a layer of the Latin cube,
then by  Property \ref{two4cycle}, 
this $4$-cycle belongs to the  intersection of two $3$-cubes.
But each $4$-cycle defines a unique $3$-cube, which implies that 
the intersection of the two $3$-cubes
contains $8$ cells. If not, these $2$ cells must have distinct row, column
and file coordinates, so if we denote these two cells by
$(i_1,j_1,k_1)$ and $(i_2,j_2,k_2)$, respectively, then
$i_1 \neq i_2$, $j_1 \neq j_2$, $k_1 \neq k_2$. Hence, the intersection of 
the two $3$-cubes must be the
$8$ cells $(i_1,j_1,k_1), $ $ (i_1,j_2,k_1), $ $ (i_2,j_1,k_1), $ $ (i_2,j_2,k_1),$ $ (i_1,j_1,k_2), $ 
$(i_1,j_2,k_2), $ $ (i_2,j_1,k_2), $ $  (i_2,j_2,k_2)$.
\end{proof}

\begin{definition}
Given a $3$-cube $$\mathcal{C}=\{(i_1,j_1,k_1), (i_1,j_2,k_1), 
(i_2,j_1,k_1), (i_2,j_2,k_1), (i_1,j_1,k_2), (i_1,j_2,k_2), 
(i_2,j_1,k_2),  (i_2,j_2,k_2)\}$$ in a Latin cube $L$, 
a \textit{swap on $\mathcal{C}$} (or simply a \textit{swap})
denotes the transformation $L \rightarrow L'$ which retains the content
of all cells of $L$ except that if 
$$L(i_1,j_1,k_1)=L(i_2,j_2,k_1)=L(i_1,j_2,k_2)=L(i_2,j_1,k_2)=x_1$$
and $$L(i_1,j_2,k_1)= L(i_2,j_1,k_1)=L(i_1,j_1,k_2)= L(i_2,j_2,k_2)=x_2$$ then
$$L'(i_1,j_1,k_1)=L'(i_2,j_2,k_1)=L'(i_1,j_2,k_2)=L'(i_2,j_1,k_2)=x_2$$
and $$L'(i_1,j_2,k_1)= L'(i_2,j_1,k_1)=L'(i_1,j_1,k_2)= L(i_2,j_2,k_2)=x_1.$$
\end{definition}

\begin{property}
\label{rcfblock}
Consider an arbitrary column $\{(i_1,j_1,k_1),\dots, (i_n,j_1,k_1)\}$
of a Boolean Latin cube $B$ of order $n$. 
For any $k_2$ ($j_2$), there exists a unique $j_2$ ($k_2$),
such that $B(x,j_1,k_1)=B(x,j_2,k_2)$ for every $x \in \{1,\dots,n\}$. 
\end{property}
\begin{proof}
For any $k_2$, we can choose $j_2$ satisfying $a_{j_2}=a_{j_1}+a_{k_1}-a_{k_2}$, and
for any $j_2$, we can choose $k_2$ satisfying $a_{k_2}=a_{j_1}+a_{k_1}-a_{j_2}$.
\end{proof}
Evidently, all rows and files of a Boolean Latin cube have corresponding properties.

\begin{property}
\label{Srcfblock}
Let $B$ be a Boolean Latin cube of order $n$, 
$\mathfrak{b}$ an arbitary symbol in $B$, and
$S_1$ be the set of cells of $B$ in the first row layer which contain $\mathfrak{b}$.
For any row layer $i$, the set of cells $S_i$ of $B$ in row layer $i$ which have the same column and file coordinates as cells in $S_1$ all 
contain the same symbol. 
\end{property}
\begin{proof}
Assume $(i,j_1,k_1) \in S_i$ and $B(i,j_1,k_1)=x$, and consider an arbitrary cell $(i,j_2,k_2) \in S_i$. 
By definition, there are two cells
$(1,j_1,k_1)$ and $(1,j_2,k_2)$ such that 
$B(1,j_1,k_1) = B(1,j_2,k_2) = \mathfrak{b}$, that is,
$a_1 + a_{j_1} + a_{k_1} = a_1 + a_{j_2} + a_{k_2}$. This implies that
$a_i + a_{j_1} + a_{k_1} = a_i + a_{j_2} + a_{k_2}$, which means
that $B(i,j_2,k_2)=B(i,j_1,k_1)=x$. Hence, all cells in $S_i$ contain the same 
symbol.
\end{proof}
Note that all column and file layers of $B$ have the same property.

The following simple observation enables us to permute layers and symbols
in a Latin cube.

\begin{property}
If $L$ is a Latin cube, then
the cube obtained by permuting the row layers, the column layers, 
the file layers and/or the symbols of $L$ is a Latin cube.
\end{property}

For Boolean Latin cubes an even stronger property holds.
If a Latin cube $L'$ is obtained from another Latin cube $L$ by permuting
row/column/file layers and/or symbols of $L$, then we say
that $L$ and $L'$ are {\em isomorphic}.
Henceforth, all Latin cubes have order $n$.

\begin{property}
If $L$ is isomorphic to a Boolean Latin cube, then any cell of $L$ is in $n-1$
$3$-cubes.
Moreover,
Property \ref{traversalsetproperty}, \ref{rcfblock}, and \ref{Srcfblock}
hold for $L$. 
\end{property}

\bigskip

In the following we shall define some sets of cells in Latin cubes
that are isomorphic to Boolean Latin cubes.

\begin{definition}
Let $L$ be a Latin cube that is isomorphic to a Boolean Latin cube.
%
A \textit{row block} of $L$ is a set of $n$ rows $R_{i,k}$ such that for every pair 
of rows $R_{i_1,k_1}=\{(i_1,j,k_1): j \in [n]\}$ and 
$R_{i_2,k_2} = \{(i_2,j,k_2): j \in [n]\}$ in this set, $B(i_1,x,k_1)=B(i_2,x,k_2)$ for every $x \in \{1,\dots,n\}$. It is obvious that there are $n$ row blocks in total.
\textit{Column blocks} and  \textit{file blocks} are defined similarly.
\end{definition}

\begin{property}
\label{cellbelongsrowblock}
If 
$$\mathcal{C}=\{(i_1,j_1,k_1),  (i_1,j_2,k_1), (i_2,j_1,k_1), 
(i_2,j_2,k_1), (i_1,j_1,k_2), (i_1,j_2,k_2), (i_2,j_1,k_2), (i_2,j_2,k_2)\}$$
is a $3$-cube
in a Latin cube $L$ that is isomorphic to a Boolean cube,
then the two rows $R_{i_1,k_1}$ and $R_{i_2,k_2}$ are in the same row block,
as are also the two rows $R_{i_2, k_1}$ and $R_{i_1, k_2}$. 
\end{property}
Note that a  similar property holds for columns blocks and file blocks.

\begin{definition}
If $L$ is a Latin cube that is isomorphic to a 
Boolean Latin cube, a \textit{transversal-set} $\mathfrak{t}$ of $L$ 
is a set of $n$ cells that satisfy the following

\begin{itemize}

	\item no two cells in $\mathfrak{t}$ are in the same row/column/file;
	
	 \item no two cells in $\mathfrak{t}$ contain the same symbol;
	
	\item for any two cells in $\mathfrak{t}$, there is a unique $3$-cube
	that contain these cells.

\end{itemize}
\end{definition}
Note that by Property \ref{traversalsetproperty}, a transversal-set is well-defined,
and every row block, column block and file block contains exactly $n$ disjoint 
transversal-sets.

Based on Property \ref{Srcfblock}, we make the following definition.
\begin{definition}
A \textit{symbol-row block} of a Latin cube $L$ that is isomorphic to a Boolean
Latin cube is a set $\mathfrak{s}$  of $n^2$ cells 
satisfying that
\begin{itemize}
	
	\item all cells of $\mathfrak{s}$ that are in the same row layer
	contain the same symbol, and
	
	\item for every cell of $\mathfrak{s}$, there are $n-1$ other cells that have
	the same column and file coordinate.

\end{itemize}
\textit{Symbol-column blocks} and \textit{symbol-file blocks} are defined similarly.
\end{definition}

An intersection between a symbol-row block and a row layer (or a symbol-column block and a column layer, or a symbol-file block and a file layer) is called a \textit{symbol-set}. It is obvious that all cells in a symbol-set contain the same symbol, and that
each row layer, column layer, file layer, symbol-row block, symbol-column block, 
and symbol-file block contains $n$ symbol-sets.

\begin{definition}
A \textit{symbol block} of a Latin cube $L$ is a set of $n^2$ cells such that all these cells contain the same symbol. 
\end{definition}
Note that a Latin cube that is isomorphic to a Boolean Latin cube 
contains $n$ symbol blocks in total, and 
for each symbol block, there are three different ways to divide this 
symbol block to $n$ disjoint symbol-sets 
(group the symbol sets based on the row layers, the column layers or the file layers).

\bigskip

Given an $n\times n \times n$ cube $A$ where each cell contains a subset of the
symbols in $\{1,\dots,n\}$, and a Latin cube $L$ of order $n$
that does not avoid $A$, we say that those cells $(i,j,k)$ of $L$ where 
$L(i,j,k) \in A(i,j,k)$ are \textit{conflict cells of $L$ with $A$} 
(or simply {\em conflicts} of $L$).
An \textit{allowed $3$-cube} of $L$ is a $3$-cube 
$$\mathcal{C}=\{(i_1,j_1,k_1),  (i_1,j_2,k_1), (i_2,j_1,k_1), (i_2,j_2,k_1), 
(i_1,j_1,k_2),  
(i_1,j_2,k_2),  (i_2,j_1,k_2),   (i_2,j_2,k_2)\}$$ in $L$ such that swapping on $\mathcal{C}$ yields a Latin cube
$L'$ where none of $(i_1,j_1,k_1), $ $ (i_1,j_2,k_1), $ $ (i_2,j_1,k_1), $ $ (i_2,j_2,k_1),$ $ (i_1,j_1,k_2), $ 
$(i_1,j_2,k_2), $ $ (i_2,j_1,k_2), $ $  (i_2,j_2,k_2)$ is a conflict.

\section{Proof of the main theorem}



In this section we prove Theorem \ref{maintheorem}. 
Our basic proof strategy is similar to the one in 
\cite{Lina,AndrenCasselgrenOhman}; however, due to the extra dimension
in a Latin cube, our arguments are considerably more involved and
somewhat technical.
Our starting point in the proof is the Boolean Latin cube; we 
permute its row layers, column layers, file layers and symbols 
so that the resulting Latin cube does not have 
too many conflicts with a given $(m,m,m,m)$-cube $A$. After that, we find a set of allowed $3$-cubes such that each conflict belongs to one of them, with no two of the $3$-cubes intersecting, and swap on those $3$-cubes.

The proof of Theorem \ref{maintheorem} involves a number of parameters:
$$\alpha, \gamma, \kappa, \epsilon, \theta,$$
and a number of inequalities that they must satisfy. For the reader's convenience,
explicit choices for which the proof holds are presented here:
$$\alpha = 1-38 \times 2^{-25}, \gamma=2^{-25}, \kappa= 6 \times 2^{-25}, \epsilon = 2^{-6}, \theta=2^{-12}.$$

By an example of an unavoidable 
$(\lfloor {\frac{n}{3}} \rfloor+1, 
\lfloor {\frac{n}{3}} \rfloor+1, \lfloor {\frac{n}{3}} \rfloor+1)$-arrays
in \cite{CulterOhman}, 
the value of $\gamma$ for which Theorem \ref{maintheorem}
holds cannot exceed $\frac{1}{3}$.
Thus, since the numerical value of $\gamma$ for which the theorem holds
is not anywhere near
what we expect to be optimal, we have not put an effort into choosing optimal values
for these parameters. 
Moreover, for simplicity of notation,
we shall omit floor and ceiling signs whenever these are not crucial.

We shall establish that our main theorem holds by
proving two lemmas.

\begin{lemma}
Let $\alpha, \gamma, \kappa$ be constants and $n=2^t$ such that  

$$ \Big(7n^2 \dfrac{(\gamma n)^{\kappa n}}{(\kappa n)!} +3n^3\dfrac{{(2\gamma n)^{(1-\alpha-2\gamma)n/3}}}{((1-\alpha-2\gamma)n/3)!}\Big) < 1.$$
For any $(\gamma n, \gamma n, \gamma n, \gamma n)$-cube $A$ of order $n$ there is a quadruple of permutations $\sigma=(\tau_1, \tau_2, \tau_3, \tau_4)$
of the row layers, the column layers, the file layers and the symbols of the Boolean Latin cube $B$ of order $n$,
respectively, 
such that applying $\sigma$ to $B$, we obtain a Latin cube $L$ satisfying
the following:
\begin{itemize}
\item[(a)] No row in $L$ contains more than $\kappa n$ conflicts with $A$.
\item[(b)] No column in $L$ contains more than $\kappa n$ conflicts with $A$.
\item[(c)] No file in $L$ contains more than $\kappa n$ conflicts with $A$.
\item[(d)] No symbol-set in $L$ contains more than $\kappa n$ conflicts with $A$.
\item[(e)] No transversal-set in $L$ contains more than $\kappa n$ conflicts with $A$.
\item[(f)] Each cell of $L$ belongs to at least $\alpha n$ allowed 3-cubes.
\end{itemize}
\end{lemma}

\begin{proof}
Let $X_a$, $X_b$, $X_c$, $X_d$, $X_e$ and $X_f$ be the number of permutations which do not fulfill the conditions
$(a)$, $(b)$, $(c)$, $(d)$, $(e)$ and $(f)$, respectively. Let $X$ be the number of permutations satisfying the 
five conditions $(a)$, $(b)$, $(c)$, $(d)$, $(e)$ and $(f)$. There are $(n!)^4$ ways to permute the row layers,
the column layers, the file layers and the symbols, so we have
$$X \geq (n!)^4 - X_a - X_b - X_c - X_d - X_e - X_f.$$
We shall prove that $X$ is greater than $0$.

\begin{itemize}
\item To estimate $X_a$, assume that for any fixed permutation 
$(\tau_1,\tau_3,\tau_4)$ of the row layers,
the file layers and the symbols, 
at most $N_a$ choices of a permutation $\tau_2$ of the column layers yield a quadruple
$(\tau_1, \tau_2, \tau_3,\tau_4)$ of permutations 
that break condition $(a)$; so $X_a \leq n!n!n!N_a$.

Let $R$ be a fixed row chosen arbitrarily; 
we count the number of ways a permutation $\tau_2$ of the column layers 
can be constructed so that $(a)$ does not hold on row $R$.
Let $S$ be a set of size $\kappa n$ of column layers of $A$.
There are $n \choose \kappa n$ ways to choose $S$. 
In order to have a conflict at cell $(i,j,k)$ of $R$, the column layers 
should be permuted
in such a way that in the resulting Latin cube $L$, $L(i,j,k) \in A(i,j,k)$. 
Since $|A(i,j,k)| \leq \gamma n$, there are at most $(\gamma n)^{\kappa n}$ 
ways to choose which column 
layers of $B$ are mapped by $\tau_2$ to column layers in $S$ so that 
all cells on row $R$ that are in $S$ are conflicts. The rest of the column 
layers can be arranged in any of
the $(n-\kappa n)!$ possible ways. In total this gives at most
$${n \choose \kappa n}(\gamma n)^{\kappa n}(n-\kappa n)! 
= \dfrac{n!(\gamma n)^{\kappa n}}{(\kappa n)!}$$
permutations $\tau_2$ that do not satisfy condition $(a)$ on row $R$.
There are $n^2$ rows in $B$, so we have 
$$N_a \leq n^2 \dfrac{n!(\gamma n)^{\kappa n}}{(\kappa n)!}$$
and $$X_a \leq n!n!n!N_a \leq n^2(n!)^4\dfrac{(\gamma n)^{\kappa n}}{(\kappa n)!}.$$

An analogous argument gives the same bound for $X_b$, $X_c$, so in total, we have that
$$X_a + X_b + X_c  \leq 3n^2(n!)^4\dfrac{(\gamma n)^{\kappa n}}{(\kappa n)!}.$$

\item To estimate $X_d$, assume that for any fixed permutation $(\tau_1,\tau_3, \tau_4)$ of the row layers, the file layers
and the symbols, at most $N_d$ choices of a permutation $\tau_2$ of the column layers give a quadruple 
$(\tau_1, \tau_2, \tau_3,\tau_4)$ of permutations 
that break condition $(d)$; then $X_d \leq n!n!n!N_d$.

Let $\mathfrak{b}$ be a fixed symbol chosen arbitrarily;
we count the number of ways a permutation $\tau_2$ of the column layers 
can be constructed so that $(d)$ does not hold for $\mathfrak{b}$
in a given row layer. 
Let $R_L$ be a fixed row layer; there are $n$ cells containing $\mathfrak{b}$ in $R_L$
and these cells belong to $n$ different column layers since $B$ is a boolean
Latin cube.
Let $S$ be a set of size $\kappa n$ of column layers of $A$; 
there are $n \choose \kappa n$ ways to choose $S$. 
Since in $A$, each symbol occurs at most $\gamma n$ time in each row, 
there are at most $(\gamma n)^{\kappa n}$ ways to choose which 
column layers of $B$ are mapped by $\tau_2$ to column layers in $S$ so that 
all cells containing $\mathfrak{b}$ on row layer $R_L$ that are
in $S$ are conflicts.
The rest of the column layers can be arranged in any of
the $(n-\kappa n)!$ possible ways. In total this gives at most
$${n \choose \kappa n}(\gamma n)^{\kappa n}(n-\kappa n)! 
= \dfrac{n!(\gamma n)^{\kappa n}}{(\kappa n)!}$$
permutations $\tau_2$ such that in the resulting Latin cube $L$, symbol $\mathfrak{b}$
appears in more than $\kappa n$ conflicts in the row layer $R_L$.
There are $n$ different row layers, $n$ different column layers and $n$ different file layers in $B$,
so we deduce that there are at most $3n\dfrac{n!(\gamma n)^{\kappa n}}{(\kappa n)!}$
permutations $\tau_2$ that do not satisfy condition $(d)$ on symbol $\mathfrak{b}$.
There are $n$ symbols in $B$, so we have
$$N_d \leq 3n^2 \dfrac{n!(\gamma n)^{\kappa n}}{(\kappa n)!}.$$
and
$$X_d \leq 3n^2 (n!)^4\dfrac{(\gamma n)^{\kappa n}}{(\kappa n)!}.$$

\item To estimate $X_e$, assume that for any fixed permutation $(\tau_1,\tau_2,\tau_3)$ of the row layers, the column layers,
the file layers, 
at most $N_e$ choices of a permutation $\tau_4$ of the symbols give a quadruple
$(\tau_1, \tau_2, \tau_3,\tau_4)$ of permutations 
that break condition $(e)$; so $X_e \leq n!n!n!N_e$.

Let $T$ be a fixed transversal-set chosen arbitrarily; 
we count the number of ways a permutation $\tau_4$ of the symbols
can be constructed so that $(e)$ does not hold on the set $T$.
Let $S$ be a set of size $\kappa n$ of cells of $T$;
there are $n \choose \kappa n$ ways to choose $S$. 
In order to have a conflict at cell $(i,j,k)$ of $T$, the symbols should be permuted
in such a way that in the resulting Latin cube $L$, $L(i,j,k) \in A(i,j,k)$. 
Since $|A(i,j,k)| \leq \gamma n$, there are at most $(\gamma n)^{\kappa n}$ 
ways to choose which symbols of $B$ are mapped by $\tau_4$ to cells in $S$ so that 
all cells in $S$ are conflicts. The rest of the symbols can be arranged in any of
the $(n-\kappa n)!$ possible ways. In total this gives at most
$${n \choose \kappa n}(\gamma n)^{\kappa n}(n-\kappa n)! 
= \dfrac{n!(\gamma n)^{\kappa n}}{(\kappa n)!}$$
permutations $\tau_4$ that do not satisfy condition $(e)$ on the transversal-set $T$.
There are $n^2$ transversal-sets in $B$, so we have 
$$N_e \leq n^2 \dfrac{n!(\gamma n)^{\kappa n}}{(\kappa n)!},$$
and so
$$X_e \leq n!n!n!N_e \leq n^2(n!)^4\dfrac{(\gamma n)^{\kappa n}}{(\kappa n)!}.$$

\item To estimate $X_f$, assume that for any fixed permutation $(\tau_2,\tau_4)$ of the  
column layers and the symbols at most $N_f$ choices of a pair $(\tau_1, \tau_3)$ 
of the row layers and the file layers yield a quadruple
$(\tau_1, \tau_2, \tau_3, \tau_4)$ of permutations 
that break condition $(f)$, then $X_f \leq n!n!N_f$.

Let $(i_1,j_1,k_1)$ be an arbitrary fixed cell of $A$. 
There are $n^2$ ways to choose a row layer $i_x$ 
and a file layer $k_x$ so that $i_1=\tau_1(i_x)$ and $k_1=\tau_3(k_x)$;
we fix such a row layer $i_x$ and file layer $k_x$.
Moreover, each $3$-cube $\mathcal{C}$ containing 
$(i_1,j_1,k_1)$ is uniquely determined by the value of $j_2 \neq j_1$ 
where $(i_1,j_2,k_1) \in \mathcal{C}$;
so a pair of permutations $(\tau_1, \tau_3)$ satisfy that
the quadruple $(\tau_1, \tau_2, \tau_3, \tau_4)$ adds to $X_f$
if and only if there are more than $(1-\alpha)n$ choices for $j_2$ so that the swap along $\mathcal{C}$
is not allowed. We shall count the number of ways of choosing $(\tau_1, \tau_3)$
so that this holds.

Let us first note that there are at most $2 \gamma n$ choices for $(i_x, k_x)$ that yield
a $3$-cube $\mathcal{C}$ that is not allowed because of a conflict on row $i_1$
in file layer $k_1$;
that is, after swapping on $\mathcal{C}$, we have a conflict cell on row $i_1$
in file layer $k_1$.
This follows from the fact that there are $\gamma n$ choices for 
$j_2$ such that $A(i_1,j_2,k_1)$ contains $L(i_1,j_1,k_1)$,
and since $|A(i_1,j_1,k_1)| \leq \gamma n$, we
have $\gamma n$ choices for $j_2$ so that $L(i_1, j_2, k_1) \in A(i_1,j_1,k_1)$.
So for a permutation $(\tau_1,\tau_3)$ to contribute to $N_f$, 
$(\tau_1,\tau_3)$ must be such that at least
$(1-\alpha -2\gamma)n$ $3$-cubes containing the cell $(i_1,j_1,k_1)$ are not allowed
because of restrictions on
rows of $A$ that are distinct from row $i_1$ in file layer $k_1$.
Since each $3$-cube $\mathcal{C}$ containing $(i_1,j_1,k_1)$ has cells from three
other rows, this implies that at least
$(1-\alpha-2\gamma)n / 3$  $3$-cubes cannot be allowed because of conflicts appearing in 
one of these rows.
 
Let $N_{f_1}$ be the number of pairs of permutations $(\tau_1,\tau_3)$ such that
at least $(1-\alpha-2\gamma)n / 3$  $3$-cubes containing $(i_1,j_1,k_1)$ 
cannot be allowed 
because swapping yields conflicts in cells in 
file layer $k_1$ that are not contained in row layer $i_1$.
Let us first note that 
there are $(n-1)!$ ways to permute the remaining file layers of $B$.
Consider a fixed permutation $\tau_3$ of the file layers; 
we count the number of permutations
$\tau_1$ of the row layers such that the pair $(\tau_1,\tau_3)$ 
contributes to $N_{f1}$. 
Let $S$ be a set of columns, ($|S|=(1-\alpha-2 \gamma)n/3$),
such that for every column $C_{j_2,k_2} \in S$, there is a unique $i_2$ satisfying
that
$$\mathcal{C}=\{(i_1,j_1,k_1), (i_1,j_2,k_1), (i_2,j_1,k_1), (i_2,j_2,k_1), 
(i_1,j_1,k_2), (i_1,j_2,k_2), (i_2,j_1,k_2), (i_2,j_2,k_2)\}$$ is a $3$-cube 
and this $3$-cube is not allowed because of conflicts arising in row $i_2$ in
file layer $k_1$. 
There are $n -1 \choose (1-\alpha-2\gamma)n / 3$ ways to choose $S$. 
Fix a column $C_{j_2,k_2} \in S$; in column $j_1$ of file layer $k_1$ 
of $A$, there are at most $\gamma n$ cells
containing $L(i_1,j_1,k_1)$ and in the column $j_2$ in file layer $k_1$ of $A$, 
there are at most $\gamma n$ 
cells containing $L(i_1,j_2,k_1)$, so there are up to $2\gamma n$ choices
for $\tau^{-1}_1(i_2)$ in $B$ that would make $\mathcal{C}$ disallowed because of 
conflicts arising in rows distinct from $i_1$ in the file layer $k_1$.

Since every column in $S$ yields a unique row index, $S$
determines $\tau_1$ on $(1-\alpha-2 \gamma)n/3$ row layers.
The remaining row layers can be permuted in $(n-1-(1-\alpha-2\gamma)n / 3)!$ ways. 
This implies that the total number of permutations $\tau_1$
that yield at least $(1-\alpha-2\gamma)n / 3$  $3$-cubes that are not 
allowed because of conflicts appearing in file layer 
$k_1$ that are not contained in row layer $i_1$
is bounded from above  by
$${{n-1} \choose {(1 - \alpha - 2\gamma)n/3}} {(2\gamma n)^{(1-\alpha-2\gamma)n/3}} {(n-1-(1-\alpha-2\gamma)n / 3)!}
= \dfrac{(n-1)! {(2\gamma n)^{(1-\alpha-2\gamma)n/3}}}{((1-\alpha-2\gamma)n/3)!}$$
Hence, $N_{f_1} \leq (n-1)! \dfrac{(n-1)! {(2\gamma n)^{(1-\alpha-2\gamma)n/3}}}{((1-\alpha-2\gamma)n/3)!}$.

Let $N_{f_2}$ be the number of pairs of permutations $(\tau_1,\tau_3)$ such that
at least $(1-\alpha-2\gamma)n / 3$  $3$-cubes
containing $(i_1,j_1,k_1)$ are not allowed because
swapping on them yields conflicts in 
rows contained in the row layer $i_1$ but not in file layer $k_1$. 
There are $(n-1)!$ ways to permute the remaining row layers of $B$.
We consider a fixed permutation $\tau_1$ of the row layers and count the number of permutations $\tau_3$
of the file layers such that the pair $(\tau_1,\tau_3)$ 
contributes to $N_{f_2}$. 
Let $S$ be a set of files, ($|S|=(1-\alpha-2 \gamma)n/3$), 
such that for every file $F_{i_2,j_2} \in S$, there is a unique $k_2$ satisfying that
$$\mathcal{C}=\{(i_1,j_1,k_1), (i_1,j_2,k_1), (i_2,j_1,k_1), (i_2,j_2,k_1),
(i_1,j_1,k_2), (i_1,j_2,k_2), (i_2,j_1,k_2), (i_2,j_2,k_2)\}$$ 
is a $3$-cube and this $3$-cube
is not allowed because of conflicts arising in cells in row layer $i_1$
that are not in file layer $k_1$. 
There are $n -1 \choose (1-\alpha-2\gamma)n / 3$ ways to choose $S$. 
Fix a file $F_{i_2,j_2} \in S$, in the file $F_{i_1,j_1}$ of $A$, 
there are at most $\gamma n$ cells
containing $L(i_1,j_1,k_1)$ and in the file $F_{i_1,j_2}$ of $A$, 
there are at most $\gamma n$
cells containing $L(i_1,j_2,k_1)$, so there are up to $2\gamma n$ choices
for $\tau^{-1}_3(k_2)$ in $B$ that would make $\mathcal{C}$ disallowed because of 
possible conflicts in row layer $i_1$ that are not in file layer $k_1$.

As before, $S$ determines how $\tau_3$ acts on 
$(1-\alpha-2 \gamma)n/3)$ file layers, and
the remaining file layers can be permuted in $(n-1-(1-\alpha-2\gamma)n / 3)!$ ways. 
This implies that the total number of permutations $\tau_3$ with not 
enough allowed $3$-cubes due to the fact that 
swapping yield conflicts in 
rows contained in the row layer $i_1$ but not in file layer $k_1$ is bounded
from above by 
$${{n-1} \choose {(1 - \alpha - 2\gamma)n/3}} {(2\gamma n)^{(1-\alpha-2\gamma)n/3}} {(n-1-(1-\alpha-2\gamma)n / 3)!}
= \dfrac{(n-1)! {(2\gamma n)^{(1-\alpha-2\gamma)n/3}}}{((1-\alpha-2\gamma)n/3)!}$$
Hence, $N_{f_2} \leq (n-1)! \dfrac{(n-1)! {(2\gamma n)^{(1-\alpha-2\gamma)n/3}}}{((1-\alpha-2\gamma)n/3)!}$.

Let $N_{f_3}$ be the number of pairs of permutations $(\tau_1,\tau_3)$ such that
at least $(1-\alpha-2\gamma)n / 3$  $3$-cubes $\mathcal{C}$ 
containing $(i_1,j_1,k_1)$ are not allowed 
because swapping on them yields conflicts in cells which lie in 
row and file layers distinct from $i_1$ and $k_1$, respectively.
There are $(n-1)!$ ways to permute the remaining file layers of $B$.
Consider a fixed permutation $\tau_3$ of the file layers; we count the number of permutations $\tau_1$
of the row layers such that the pair $(\tau_1,\tau_3)$ contributes to $N_{f_3}$. 
Let $S$ be a set of columns ($|S|=(1-\alpha-2 \gamma)n/3$),
such that for every column $C_{j_2,k_2} \in S$, there is a unique $i_2$ 
satisfying that
$$\mathcal{C}=\{(i_1,j_1,k_1), (i_1,j_2,k_1), (i_2,j_1,k_1), (i_2,j_2,k_1), 
(i_1,j_1,k_2), (i_1,j_2,k_2), (i_2,j_1,k_2), (i_2,j_2,k_2)\}$$ 
is a $3$-cube which
is not allowed because of swapping yields conflicts in 
cells in row $i_2$ in file layer $k_2$.
There are $n -1 \choose (1-\alpha-2\gamma)n / 3$ ways to choose $S$. 
Fix a column $C_{j_2,k_2} \in S$; in the column $C_{j_2,k_2}$ of $A$, 
there are at most $\gamma n$ cells
containing symbol $L(i_1,j_1,k_1)$, and in the column 
$C_{j_1,k_2}$ of $A$, there are 
at most $\gamma n$ 
cells containing $L(i_1,j_2,k_1)$; so there are up to $2\gamma n$ choices
for $\tau^{-1}_1(i_2)$ in $B$ that 
would make $\mathcal{C}$ disallowed because swapping
yields conflicts in cells which lie in 
row and file layers distinct from $i_1$ and $k_1$, respectively.

The set $S$ determines how $\tau_1$ acts on $(1-\alpha-2 \gamma)n/3$ row layers.
The remaining row layers can be permuted in 
$(n-1-(1-\alpha-2\gamma)n / 3)!$ ways. 
This implies that the total number of permutations 
$\tau_1$ with too few allowed $3$-cubes because of conflicts arising
in cells in 
row and file layers distinct from $i_1$ and $k_1$ is bounded from above by
$${{n-1} \choose {(1 - \alpha - 2\gamma)n/3}} {(2\gamma n)^{(1-\alpha-2\gamma)n/3}} {(n-1-(1-\alpha-2\gamma)n / 3)!}
= \dfrac{(n-1)! {(2\gamma n)^{(1-\alpha-2\gamma)n/3}}}{((1-\alpha-2\gamma)n/3)!}$$
Hence, $N_{f_3} \leq (n-1)! \dfrac{(n-1)! {(2\gamma n)^{(1-\alpha-2\gamma)n/3}}}{((1-\alpha-2\gamma)n/3)!}$.

The Boolean Latin cube contains $n^3$ cells in total, so
$$N_f \leq n^3(n^2 N_{f_1} + n^2 N_{f_2} + n^2 N_{f_3}) 
\leq 3n^3 (n!)^2 \dfrac{{(2\gamma n)^{(1-\alpha-2\gamma)n/3}}}{((1-\alpha-2\gamma)n/3)!}$$
and 
$$X_f \leq (n!)^2N_f
\leq 3n^3 (n!)^4 \dfrac{{(2\gamma n)^{(1-\alpha-2\gamma)n/3}}}{((1-\alpha-2\gamma)n/3)!}$$
\end{itemize}
Summing up, we conclude that
$$X \geq (n!)^4 - 7n^2 (n!)^4\dfrac{(\gamma n)^{\kappa n}}{(\kappa n)!} - 3n^3 (n!)^4 \dfrac{{(2\gamma n)^{(1-\alpha-2\gamma)n/3}}}{((1-\alpha-2\gamma)n/3)!}$$
$$\geq (n!)^4 \Big(1 - 7n^2 \dfrac{(\gamma n)^{\kappa n}}{(\kappa n)!} - 3n^3\dfrac{{(2\gamma n)^{(1-\alpha-2\gamma)n/3}}}{((1-\alpha-2\gamma)n/3)!}\Big)$$
By assumption, $X$ is strictly greater than $0$.
\end{proof}

\begin{lemma}
Let $L$ be a Latin cube that is isomorphic to a Boolean Latin cube,
and let $A$ be an $(m,m,m,m)$-cube; both of order $n$.
Furthermore, let $\alpha, \gamma, \kappa, \theta, \epsilon$ be 
constants,  $n=2^t$ such that $\epsilon n \geq 3$ and 

$$\alpha n-21\kappa n -7\epsilon n-\dfrac{84\kappa}{\epsilon}n-\dfrac{21\theta}{\epsilon} n-\dfrac{80\kappa}{\theta} n-28>0$$

If $L$ has the following properties:
\begin{itemize}
\item[(a)] no row in $L$ contains more than $\kappa n$ conflicts with $A$;
\item[(b)] no column in $L$ contains more than $\kappa n$ conflicts with $A$;
\item[(c)] no file in $L$ contains more than $\kappa n$ conflicts with $A$;
\item[(d)] no symbol-set in $L$ contains more than $\kappa n$ conflicts with $A$;
\item[(e)] no transversal-set in $L$ contains more than $\kappa n$ conflicts 
with $A$;
\item[(f)] each cell of $L$ belongs to at least $\alpha n$ allowed 3-cubes;
\end{itemize}
then there is a set of disjoint allowed $3$-cubes such that each conflict 
of $L$ belongs to one of them. Thus, by performing a number 
of swaps on $3$-cubes in $L$, 
we obtain a Latin cube $L'$ that avoids $A$.
\end{lemma}
\begin{proof}
For constructing $L'$ from $L$, we will perform a number of swaps on $3$-cubes,
and we shall refer to this procedure as
\textit{$S$-swap}.  We are going to construct a set $S$ of disjoint allowed $3$-cubes such that each conflict of $L$ with $A$ belongs to one of them. A cell that belongs to a $3$-cube in $S$ is called \textit{used} in $S$-swap. 
Since no row in $L$ contains more than $\kappa n$ conflicts with $A$, there are at most $\kappa n^3$ conflicts in $L$, which implies 
that the total number of cells that are used in 
$S$-swap is at most $8 \kappa n^3$.

A row layer, a column layer, a file layer, a row block, a column block, a file block, a symbol block, a symbol-row block, a symbol-column block, 
or a symbol-file block is \textit{overloaded} if such a layer or block
contains at least $\theta n^2$ cells that are used in $S$-swap; note that no more than $\dfrac{8\kappa n^3}{\theta n^2}=\dfrac{8\kappa}{\theta} n$ 
layers or blocks of each type
are $S$-overloaded. A row, a column, a file, a transversal-set, 
or a symbol-set is \textit{overloaded} if 
this row, column, file, transversal-set or symbol-set contains 
at least $\epsilon n$ cells that are used in $S$-swap.

Using these facts, let us now construct our set $S$ by steps; at each step we consider a conflict cell $(i_1,j_1,k_1)$ and include an allowed $3$-cube containing $(i_1,j_1,k_1)$ in $S$. Initially, the set $S$ is empty.

So let us consider a each conflict cell $(i_1,j_1,k_1)$ in $L$; there are at least $\alpha n$ allowed $3$-cubes containing $(i_1,j_1,k_1)$. We choose an allowed $3$-cube 
$$\mathcal{C}=\{(i_1,j_1,k_1),  (i_1,j_2,k_1),  (i_2,j_1,k_1), 
(i_2,j_2,k_1),  (i_1,j_1,k_2), (i_1,j_2,k_2), (i_2,j_1,k_2), 
(i_2,j_2,k_2)\}$$ that satisfies the following:
\begin{itemize}

\item[(1)] The row layer $i_2$, the column layer $j_2$, the file layer $k_2$, the row block containing the row $R_{i_2,k_1}$,
the column block containing the column $C_{j_2,k_1}$, 
the file block containing the file $F_{i_1,j_2}$, the symbol-row block
containing two cells $(i_1,j_2,k_1)$ and $(i_1,j_1,k_2)$, the symbol-column block containing two cells 
$(i_2,j_1,k_1)$ and $(i_1,j_1,k_2)$, the symbol-file block containing two cells $(i_1,j_2,k_1)$ and $(i_2,j_1,k_1)$, the symbol block containing symbol $L(i_1,j_2,k_1)$ are not overloaded. This eliminates at most $\dfrac{10 \times 8\kappa}{\theta} n=\dfrac{80\kappa}{\theta} n$ choices.

With this strategy for including $3$-cubes in $S$, after completing the construction of $S$, every layer (or block) contains at most 
$4\kappa n^2 + (\theta n^2-1) +4$ cells that are used in $S$-swap. Hence, the number of overloaded rows (overloaded columns, overloaded files, overloaded transversal-sets or overloaded symbol-sets) in each layer (or block)  is at most $\dfrac{4\kappa n^2 + \theta n^2 +3}{\epsilon n} \leq \dfrac{4\kappa + \theta}{\epsilon} n+1$. 
Note that here the statement ``each symbol block contains at most 
$\dfrac{4\kappa + \theta}{\epsilon} n+1$ overloaded symbol-sets''
is to be taken with respect to either row layers, column layers
or file layers, i.e., 
when we consider the $n$ different symbol sets of a given symbol
block belonging  to $n$ different row layers
(or $n$ different column layers or $n$ different file layers), 
the number of overloaded such symbol-sets is at most
$\dfrac{4\kappa + \theta}{\epsilon} n+1$.

\item[(2)] Some rows, columns, files, transversal-sets, symbol-sets are not overloaded as the following:
\begin{itemize}
\item[(2a)] The columns $C_{j_2,k_1}$, $C_{j_1,k_2}$, $C_{j_2,k_2}$ 
are not overloaded; 
this eliminates at most $\dfrac{12\kappa + 3\theta}{\epsilon} n+3$ choices
since in the file layer $k_1$ (which contains the column 
$C_{j_2,k_1}$) and in the column layer $j_1$ 
(which contains the column $C_{j_1,k_2}$) 
and in the column block which contains the column $C_{j_1,k_1}$ 
(which also contains the column $C_{j_2,k_2}$), 
there are in total 
at most $\dfrac{4\kappa + \theta}{\epsilon} n+1$ overloaded columns.
Similarly, we need that the  rows $R_{i_2,k_1}$, $R_{i_1,k_2}$, 
$R_{i_2,k_2}$ and the files $F_{i_1,j_2}$, $F_{i_2,j_1}$,
$F_{i_2,j_2}$ are not overloaded; 
this eliminates at most $\dfrac{24\kappa + 6\theta}{\epsilon} n+6$ choices.

\item[(2b)] The transversal-set $\mathfrak{t}_1$ containing $(i_2,j_1,k_1)$ 
and $(i_1,j_2,k_2)$ is not overloaded;
this eliminates at most $\dfrac{4\kappa + \theta}{\epsilon} n+1$ choices,
since in the column block
which contains the column $C_{j_1,k_1}$ 
(which also contains the transversal-set $\mathfrak{t}_1$), there are at 
most $\dfrac{4\kappa + \theta}{\epsilon} n+1$ overloaded transversal-sets. Similarly, we need that the transversal-set containing $(i_1,j_2,k_1)$ and 
$(i_2,j_1,k_2)$, and the transversal-set containing $(i_2,j_2,k_1)$ 
and $(i_1,j_1,k_2)$ are not overloaded;
this eliminates at most $\dfrac{8\kappa + 2\theta}{\epsilon} n+2$ choices.

\item[(2c)] The symbol-set $\mathfrak{s}_1$ containing $(i_2,j_1,k_2)$ and $(i_1,j_2,k_2)$ is not overloaded;
this eliminates at most $\dfrac{4\kappa + \theta}{\epsilon} n+1$ choices,
since in the symbol block which contains $(i_1,j_1,k_1)$ (which also contains the symbol-set $\mathfrak{s}_1$), there are 
at most $\dfrac{4\kappa + \theta}{\epsilon} n+1$ overloaded symbol-sets. Similarly, we need that the symbol-set
containing $(i_2,j_2,k_1)$ and $(i_1,j_2,k_2)$, and
the symbol-set containing $(i_2,j_2,k_1)$ 
and $(i_2,j_1,k_2)$ are not overloaded,
this eliminates at most $\dfrac{8\kappa + 2\theta}{\epsilon} n+2$ choices.
 
\item[(2d)] The symbol-set $\mathfrak{s}_2$ containing 
$(i_1,j_2,k_1)$ and $(i_2,j_1,k_1)$, 
and the symbol-set $\mathfrak{s}_3$ containing 
$(i_1,j_2,k_1)$ and $(i_2,j_2,k_2)$ are not overloaded. 
This eliminates at most $\dfrac{8\kappa + 2\theta}{\epsilon} n+2$ 
choices, since in the file layer $k_1$ 
(which contains the symbol-set $\mathfrak{s}_2$), and in the symbol-column block which contains $(i_1,j_1,k_1)$ (which also contains symbol-set $\mathfrak{s}_3$), there are at most $\dfrac{4\kappa + \theta}{\epsilon} n+1$ overloaded symbol-sets. Similarly, we need that the symbol-set
containing $(i_1,j_1,k_2)$ and $(i_1,j_2,k_1)$, 
the symbol-set containing $(i_1,j_1,k_2)$ and $(i_2,j_2,k_2)$,
the symbol-set containing $(i_2,j_1,k_1)$ and $(i_1,j_1,k_2)$, the symbol-set containing $(i_2,j_1,k_1)$ and $(i_2,j_2,k_2)$ are not overloaded. This eliminates at most $\dfrac{16\kappa + 4\theta}{\epsilon} n+4$ choices.

\end{itemize}
So in total, this eliminates at most $\dfrac{84\kappa + 21\theta}{\epsilon} n+21$ choices.
Note that with this strategy for including $3$-cubes in $S$, after completing the construction of $S$, every row,
column, file, transversal-set, and symbol-set contains at most $2 \kappa n +(\epsilon n-1)+2$ or $2 \kappa n +\epsilon n+1$ cells that are used in $S$-swap.

\item[(3)] Except for $(i_1,j_1,k_1)$, none of the cells in $\mathcal{C}$ are conflicts or used before in $S$-swap.
\begin{itemize}

\item[(3a)] The cell $(i_2,j_1,k_1)$ is not conflict and has not been
used before in $S$-swap; 
this eliminates at most $3\kappa n +\epsilon n+1$ choices since the column 
$C_{j_1,k_1}$ contains at most $\kappa n$ conflict cells and at most $2 \kappa n +\epsilon n+1$ cells that are used in $S$-swap. 
Similarly, we need that the cell $(i_1,j_2,k_1)$ and the cell $(i_1,j_1,k_2)$ 
are not conflicts and has not used before in $S$-swap; in total, 
this eliminates at most $6\kappa n +2\epsilon n+2$ choices.

\item[(3b)] The cell $(i_1,j_2,k_2)$ is not conflict and has not been
used before in $S$-swap. This eliminates at most 
$3\kappa n +\epsilon n+1$ choices, since in the symbol-set 
in row layer $i_1$ that contains the cell
$(i_1,j_1,k_1)$, there are at most $\kappa n$ conflict cells and at most $2 \kappa n +\epsilon n+1$ cells that have been used in $S$-swap. 
Similarly, we need that the cell 
$(i_2,j_1,k_2)$ and the cell $(i_2,j_2,k_1)$ are not conflicts and has not been
used before in $S$-swap; in total, this eliminates at most 
$6\kappa n +2\epsilon n+2$ choices.

\item[(3c)] The cell $(i_2,j_2,k_2)$ is not conflict and has not been
used before in $S$-swap. This eliminates at most $3\kappa n +\epsilon n+1$ choices since in the transversal-set containing the cell $(i_1,j_1,k_1)$, there are at most $\kappa n$ conflict cells and at most $2 \kappa n +\epsilon n+1$ cells that are used in $S$-swap. 
\end{itemize}
So in total, this eliminates at most $21\kappa n +7\epsilon n+7$ choices. 
\end{itemize}

It follows that we have at least
$$\alpha n-21\kappa n -7\epsilon n-\dfrac{84\kappa}{\epsilon}n-\dfrac{21\theta}{\epsilon} n-\dfrac{80\kappa}{\theta} n-28$$
choices for an allowed $3$-cube $\mathcal{C}$ which contains $(i_1,j_1,k_1)$. By assumption, this expression is greater than zero, so we can conclude that there is a $3$-cube satisfying these conditions. Thus we may construct the set $S$ by iteratively adding disjoint allowed $3$-cubes such that each $3$-cube contains a conflict cell.

After this process terminates, we have a set $S$ of disjoint $3$-cubes; 
we swap on all $3$-cubes in $S$ to obtain the Latin cube $L'$. 
Hence, we conclude that we can obtain a Latin cube $L'$ that avoids $A$.
\end{proof}

\end{document}